\documentclass{amsproc}
\usepackage[english]{babel}
\makeatletter
\renewcommand{\l@section}{\@tocline{1}{0pt}{0pc}{}{}}
\renewcommand{\l@subsection}{\@tocline{2}{0pt}{0pc}{}{}}
\renewcommand{\tocsection}[3]{%
  \indentlabel{\@ifnotempty{#2}{\makebox[3em][l]{\ignorespaces#1 #2.}}}#3%
}

\makeatother
\usepackage{amsfonts}
\usepackage{amsmath}
\usepackage{amssymb}
\usepackage{amsthm}
\usepackage{eucal}
\usepackage{booktabs}
\usepackage{graphicx}
\usepackage[usenames,dvipsnames,svgnames,table]{xcolor}
\definecolor{badgerred}{rgb}{0.715,0.004,0.004}
\definecolor{burntorange}{rgb}{0.801,0.332,0.0}
\usepackage[font=small,labelfont={bf,sf}, textfont={sf},margin=0em]{caption}
\usepackage{hyperref}
\hypersetup{colorlinks,
  filecolor=black,
  linkcolor=MidnightBlue,
  citecolor=NavyBlue,
  urlcolor=RoyalBlue,
  bookmarksopen=true}

\usepackage[utf8]{inputenc}

\newcommand{\R}{\mathbb{R}}

\newcommand{\hN}{{}^*\mathbb{N}}

\newcommand{\hR}{{}^*\mathbb{R}}

\newcommand{\st}{\mathop{\mathrm{St}}}
\newtheorem{theorem}{Theorem}[subsection]
\newtheorem{lemma}[theorem]{Lemma}
\newtheorem{definition}[theorem]{Definition}
\begin{document}

\title[Nonstandard Reaction Diffusion]{Nonstandard existence proofs\\
  for reaction diffusion equations} \author{Connor Olson} \address{C.O.:
  Department of Mathematics, University of Wisconsin -- Madison}

\author{Marshall Mueller} \address{M.M.: Department of Mathematics,
Tufts University -- 503 Boston Avenue, Medford, MA 02155}

\author{Sigurd B.~Angenent} \address{S.B.A.: Department of Mathematics,
  University of Wisconsin -- Madison} \date{\today}


\begin{abstract}
  We give an existence proof for distribution solutions to a scalar reaction
  diffusion equation, with the aim of illustrating both the differences and the
  common ingredients of the nonstandard and standard approaches.  In particular,
  our proof shows how the operation of taking the standard part of a nonstandard
  real number can replace several different compactness theorems, such as
  Ascoli's theorem and the Banach--Alaoglu theorem on weak$^*$-compactness of
  the unit ball in the dual of a Banach space.
\end{abstract}
\maketitle
\tableofcontents

\section{Introduction}

\subsection{Reaction diffusion equations}
We consider the Cauchy problem for scalar reaction diffusion equations of the
form
\begin{subequations}
  \begin{equation}\label{eq-rde}
    \frac{\partial u}{\partial t} = D \frac{\partial^2 u}{\partial x^2} + f(u(x, t)), \qquad
    x\in \R, t\geq0
  \end{equation}
  with prescribed initial condition
  \begin{equation}\label{eq-initial-condition}
    u(x, 0) = u_0(x).
  \end{equation}
\end{subequations}
In the setting of reaction diffusion equations the function $u(x,t)$ represents
the density at location $x\in\R$ and time $t\geq0$ of some substance which
diffuses, and simultaneously grows or decays due to chemical reaction,
biological mutation, or some other process.  The term
$D\partial^2u/\partial x^2$ in the PDE \eqref{eq-rde} accounts for the change in 
$u$ due to diffusion, while the nonlinear term $f(u)$ accounts for the reaction
rates.  The prototypical example of such a reaction diffusion equation is the
Fisher/KPP equation (see \cite{1937KPP}, \cite{1937Fisher}) in which the
reaction term is given by $f(u) = u-u^2$.

The Cauchy problem for the reaction diffusion equation \eqref{eq-rde} is to find
a function $u:\R\times[0, \infty)\to \R$ that satisfies the partial differential
equation \eqref{eq-rde} as well as the initial condition
\eqref{eq-initial-condition}.  This is a classical problem, and the existence of
such solutions is well known (see for example \cite{henry1981geometric} or
\cite{pazy2012semigroups}).  As various techniques for constructing solutions
are known, including the use of finite difference approximations to construct
solutions (see \cite[Ch.7.2]{john1991partial}), our main goal is not to give
another existence proof.  Instead, we were inspired by several introductory
texts on nonstandard analysis (notably, Keisler's undergraduate calculus text
\cite{keisler2012elementary}, Goldblatt's more advanced introduction to the
hyperreals \cite{GoldblattHyperrealLectures}, Nelson's ``radically elementary
approach'' to probability \cite{NelsonRadicallyElementary}, as well as Terry
Tao's blog post \cite{TerryTaoBlogpost}), and wanted to see what some standard
existence proofs would look like in the language of nonstandard analysis.  

In \cite{keisler2012elementary}
Keisler presents a proof of Peano's existence theorem for solutions to ordinary
differential equations
\begin{equation}\label{eq-ode}
  \frac{dx}{dt} = f(t, x(t)), \quad x(0) = x_0,
\end{equation}
using nonstandard analysis.
One possible standard proof of Peano's theorem proceeds by constructing the
numerical approximation to the solution by solving Euler's method for any small
step size $\Delta t>0$, i.e., one defines numbers $x_{i, \Delta t}$ by setting
$x_{0, \Delta t } = x_0$, and then inductively solving
\begin{equation}
  \label{eq-Euler}
  \frac{x_{i+1, \Delta t} - x_{i, \Delta t}}{\Delta t} 
  = f(i\Delta t, x_{i, \Delta t}), \quad i=0,1,2,\dots
\end{equation}
The function $x_{\Delta t} : [0, \infty) \to \R$ obtained by linearly
interpolating between the values $x_{\Delta t}(i\Delta t) = x_{i, \Delta t}$ is
Euler's numerical approximation to the solution of the differential equation
\eqref{eq-ode}.  The standard analysis proof of Peano's existence theorem then
uses Ascoli's compactness theorem to extract a sequence of step sizes
$\Delta t_n \to 0$ such that of the approximate solutions $x_{\Delta t_n}(t)$
converge uniformly to some function $\tilde x:[0, \infty)\to\R$, and concludes
by showing that the limit $\tilde x$ is a solution to the differential equation
\eqref{eq-ode}.

The nonstandard proof in Keisler's text \cite{keisler2012elementary} follows the
same outline, but one notable feature of this proof is that instead of using
Ascoli's theorem, one ``simply'' chooses the step size $\Delta t$ to be a
positive infinitesimal number.  The approximate solution then takes values in
the hyperreals, and instead of applying a compactness theorem (Ascoli's in this
case), one ``takes the standard part'' of the nonstandard approximate solution.
The proof is then completed by showing that the function that is obtained
actually satisfies the differential equation.

The standard and nonstandard proofs have some common ingredients.  In both
proofs one must find suitable estimates for the approximate solutions
$x_{i, \Delta t}$, where the estimates should not depend on the step size
$\Delta t$. Namely, the approximate solutions $x_{\Delta t}$ should be uniformly
bounded, and they should be uniformly Lipschitz continuous
($|x_{\Delta t}(t) - x_{\Delta t} (s)|\leq L |t - s|$ for all
$t, s\in [0, \infty)$, $\Delta t>0$).  In the standard proof these estimates
allow one to use Ascoli's theorem; in the nonstandard proof they guarantee that
the standard part of the approximating solution with infinitesimally small
$\Delta t$ still defines a continuous function on the standard reals.

There appear to be two main differences between the standard and nonstandard
proofs.  The first, very obviously, is that the nonstandard setting allows one
to speak rigorously of infinitely small numbers, and thereby avoid the need to
consider limits of sequences.  The second difference is that the process of
``taking the standard part'' of a hyperreal number acts as a replacement for one
compactness theorem or another: in the nonstandard proof of Peano's theorem one
avoids Ascoli's theorem by taking standard parts.  This too is probably
well-known in some circles (Terry Tao makes the point in his blog post
\cite{TerryTaoBlogpost}), but is not as obviously stated in the nonstandard
analysis texts we have seen.

In this paper we intend to further illustrate this point by proving an existence
theorem for weak, or distributional solutions of certain partial differential
equations that is analogous to the proof of Peano's theorem sketched above (see
section~\ref{sec-distribution-def} below for a very short summary of the theory
of distributions.)  Thus, to ``solve'' the reaction diffusion equation
\eqref{eq-rde} we choose space and time steps $\Delta x>0$ and $\Delta t>0$, and
discretize the PDE by replacing the second derivative with a second difference
quotient, and the time derivative with a forward difference quotient, resulting
in a finite difference equation
\begin{equation}\label{eq-finite-difference}
  \frac{u(x, t+\Delta t) - u(x, t)}{\Delta t}
  =
  D \,\frac{u(x+\Delta x, t) - 2u(x,t) + u(x-\Delta x, t)}{(\Delta x)^2}
  + f(u(x, t)).
\end{equation}
This kind of discretization is very common in numerical analysis (see
\cite{LevequeFiniteDifferenceTextbook}, or \cite{press2007numerical}) For given
initial data (but no boundary data) one can use this difference equation to
inductively compute the values of $(x,t)$ for all $(x, t)$ in a triangular grid
(see Figure~\ref{fig-grid}).

Given a solution of the difference equation one can define a generalized
function, or distribution
\begin{equation}
  \label{eq-distribution-from-finite-diff-scheme}
  \langle U, \varphi\rangle \stackrel{\rm def}=
  \sum_t \sum_x U(x, t)\varphi(x,t)\, .
\end{equation}
In a standard existence proof of weak solutions to the equation one would now
use a compactness theorem to extract a sequence
$(\Delta x_i, \Delta t_i)\to (0, 0)$ for which the corresponding distributions
$U_i$ converge in the sense of distributions, and then show that the limiting
distribution satisfies the PDE \eqref{eq-rde}.  The compactness theorem that is
required in this proof is the Banach-Alaoglu theorem about weak$^*$-compactness
in duals of Banach spaces (in our case, $L^\infty(\R^2)$ which is the Banach
space dual of $L^1(\R^2)$).

The nonstandard proof, which we give in this paper, avoids the compactness
theorem (or notions of Lebesgue integration required to define $L^\infty$) by
letting $\Delta x$ and $\Delta t$ be infinitesimal positive hyperreals, and by
taking the standard part of the expression on the right in
\eqref{eq-distribution-from-finite-diff-scheme}.  In both the standard and
nonstandard setting this approach works for the linear heat equation, i.e.~in
the case where the reaction term $f(u)$ is absent (i.e.~$f(u)\equiv 0$).  The
nonlinear case is a bit more complicated because there is no adequate definition
of $f(u)$ when $u$ is a distribution rather than a point-wise defined function.
In both the standard and nonstandard proofs we overcome this by proving that the
approximating functions are H\"older continuous, so that $f(u(x,t))$ can be
defined.  In the standard proof this again allows one to use Ascoli-Arzela and
extract a convergent subsequence.  However, since the domain $\R^2$ is not
compact, Ascoli-Arzela cannot be applied directly, and the standard proof
therefore requires one to apply the compactness theorem on an increasing
sequence of compact subsets $K_i\subset \R^2$, after which Cantor's diagonal
trick must be invoked to get a sequence of functions that converges uniformly on
every compact subset of $\R^2$.  As we show below, these issues do not come up
in the nonstandard proof.

\subsection{Comments on nonstandard analysis}

We will not even try to give an exposition of nonstandard analysis in this
paper, and instead refer the reader to the many texts that have been written on
the subject (e.g., a very incomplete
list:\cite{keisler2007foundations,keisler2012elementary,GoldblattHyperrealLectures,NelsonRadicallyElementary,AlbeverioNonstandardMethods,TerryTaoBlogpost}).

There are a few different approaches to using the hyperreals.  Keisler, in his
undergraduate nonstandard calculus text \cite{keisler2012elementary} gives an
axiomatic description of the hyperreals and their relation with the standard
reals.  In this approach functions that are defined for standard reals
automatically extend to the hyperreals, according to the \emph{transfer
  principle.}    A different
approach that also begins with an axiomatic description of the hyperreals can be
found in Nelson's ``radically elementary'' treatment of probability theory
\cite{NelsonRadicallyElementary}.

Our point of view in this paper is that of internal set theory as explained in
Goldblatt's book \cite{GoldblattHyperrealLectures} (see also Keisler's
``instructor's guide'' \cite{keisler2007foundations} to his calculus text).
Goldblatt explains the construction of the hyperreals using non principal ultra
filters (which can be thought of as analogous to the construction of real
numbers as equivalence classes of Cauchy sequences of rational numbers).  He
then extends this construction and defines internal sets, internal functions,
etc.

\subsection*{Acknowledgement}
SBA has enjoyed and benefited from conversations about nonstandard analysis with
members of the logic group at Madison, and would in particular like to thank
Terry Millar for his relentless advocacy of matters both infinitesimal and
unlimited.

\section{Distribution solutions}

\subsection{The definition of distributions}
\label{sec-distribution-def}

We recall the definition of a ``generalized function,'' i.e.~of a distribution,
which can be found in many textbooks on Real Analysis, such as Folland's book
\cite{FollandRealAnalysis}.

A real valued function $f$ on $\R^2$ is traditionally defined by specifying its
values $f(x, y)$ at each point $(x, y)\in\R^2$.  In the Theory of Distributions
a generalized function $f$ is defined by specifying its weighted averages
\begin{equation}\label{eq-function-as-distribution}
  \langle f, \varphi\rangle = \int_{\R^2} f(x, y) \varphi(x,y) \, dx\,dy
\end{equation}
for all so-called ``test functions'' $\varphi$.  A test function is any function
$\varphi:\R^2\to\R$ that is infinitely often differentiable, and which vanishes
outside a sufficiently large ball $B_R = \{(x,y)\in\R^2 : x^2+y^2 < R\}$ whose
radius $R$ is allowed to depend on the particular test function.  The set of all
test functions, which is denoted by $C_c^\infty (\R^2)$, or sometimes by
$\mathcal D(\R^2)$, is an infinite dimensional vector space.  \emph{By
  definition,} a distribution is any linear functional
$T: C_c^\infty(\R^2) \to \R$.  The most common notation for the value of a
distribution $T$ applied to a test function $\varphi$ is
$\langle T, \varphi\rangle$.  For instance, if $f:\R^2\to\R$ is a continuous
function, then the equation \eqref{eq-function-as-distribution} defines $f$ as a
distribution.  The canonical example of a distribution that does not correspond
to a function $f$ is the Dirac delta function, which is defined by
\[
  \langle \delta, \varphi \rangle \stackrel{\rm def}= \varphi(0,0).
\]
The full definition of a distribution $T$ includes the requirement that the
value $\langle T,\varphi\rangle$ depend continuously on the test function
$\varphi$.  To state this continuity condition precisely one must introduce a
notion of convergence in the space of test functions $C_c^\infty(\R^2)$.  We
refer the reader to \cite{FollandRealAnalysis} for the details, and merely
observe that a sufficient condition for a linear functional
$\varphi \mapsto \langle T, \varphi\rangle$ to be a distribution is that there
exist a constant $C$ such that
\begin{equation}\label{eq-T-bounded-by-int-phi}
  \left|\langle T,\varphi\rangle \right| \leq C \iint_{\R^2} |\varphi(x, y)| \, dx\, dy
\end{equation}
holds for all test functions $\varphi$.  Alternatively, if a constant $C$ exists
such that
\begin{equation}
  \label{eq-T-bounded-by-sup}
  \left|\langle T,\varphi\rangle \right| \leq C \sup_{(x, y)\in\R^2} |\varphi(x, y)|
\end{equation}
holds for all $\varphi\in C_c^\infty(\R^2)$, then $T$ also satisfies the
definition of a distribution.  The conditions \eqref{eq-T-bounded-by-int-phi}
and \eqref{eq-T-bounded-by-sup} are not equivalent: either one of these implies
that $T$ is a distribution.

\subsection{Distributions defined by nonstandard functions on a grid}
Let $dx, dy$ be two positive infinitesimal hyperreal numbers, and let $N$, $M$
by to positive hyperintegers such that $Ndx$ and $Mdy$ are unlimited.  Consider
the rectangular grid
\begin{equation}\label{eq-G-defined}
  G = \left\{ (k\,dx , l\, dy) \in \hR^2 \mid
    k, l\in \hN, |k|\leq N, |l|\leq M
  \right\}
\end{equation}
From the point of view of nonstandard analysis and internal set theory, $G$ is a
hyperfinite set, and for any internal function $f:G\to\hR$ there is an
$(x,y)\in G$ for which $f(x, y)$ is maximal.
\begin{lemma}
If $g:\R^2\to\R$ is a continuous function with compact support, then
\[
  \int_{\R^2} g(x, y)\,dx\,dy \approx \sum _{(x,y)\in G} g(x, y) \,dx\,dy,
\]
where $x\approx y$ means that $x-y$ is infinitesimal.
\end{lemma}
\begin{proof}
The statement of the lemma is very close to the nonstandard definition of the
Riemann integral of a continuous function, the only difference being that we are
integrating over the unbounded domain $\R^{2}$ rather than a compact rectangle
$[-\ell,\ell]\times[-\ell, \ell]\subset\R^{2}$.  Since the function $g$ has
compact support, there is a real $\ell>0$ such that $g(x, y)=0$ outside the
square $[-\ell,\ell]\times[-\ell,\ell]$.  By definition we then have
\[
  \int_{\R^{2}} g(x, y)\,dx\,dy = \int_{-\ell}^{\ell}\int_{-\ell}^{\ell} g(x,
  y)\,dx\,dy.
\]
Choose hyperintegers $L, L'\in\hN$ for which $Ldx\leq \ell <(L+1)dx$ and
$L'dy\leq \ell <(L'+1)dy$.  Then the nonstandard definition of the Riemann
integral implies
\[
  \int_{-\ell}^{\ell}\int_{-\ell}^{\ell} g(x, y)\,dx\,dy \approx
  \sum_{k=-L}^{L}\sum_{l=-L'}^{L'} g(kdx, ldy)\,dx\,dy.
\]
Finally, if $(x, y)\in G$ then $g(x, y)=0$ unless $|x|\leq\ell$ and
$|y|\leq \ell$, so that
\[
  \sum_{k=-L}^{L}\sum_{l=-L'}^{L'} g(kdx, ldy)\,dx\,dy = \sum_{(x,y)\in G} g(x,
  y)\,dx\,dy.
\]
\end{proof}

\begin{lemma}\label{lem-Tf-defined}
Suppose that $f:G\to\hR$ is a hyperreal valued function which is bounded, in the
sense that there exists a limited $C>0$ such that $|f(x,y)|\leq C$ for all
$(x,y)\in G$.  Then the expression
\begin{equation}\label{eq-Tf-defined}
  \langle T_f, \varphi\rangle \stackrel{\rm def}=
  \st \left(\sum_{(x, y)\in G} f(x, y) \varphi(x,y) \, dx\, dy\right)
\end{equation}
defines a distribution on $\R^2$.

If the function $f$ is the nonstandard extension of a (standard) continuous
function $f:\R^2\to\R$, then the distribution $T_f$ coincides with the
distribution defined by \eqref{eq-function-as-distribution}.
\end{lemma}

\begin{proof}
We first verify that the distribution is well defined.  Since $|f(x, y)|\leq C$
for all $(x,y)$ we have
\[
  \left|\sum_{(x, y)\in G} f(x, y) \varphi(x,y) \, dx\, dy\right| \leq C
  \sum_{(x, y)\in G} |\varphi(x,y)| \, dx\, dy \approx C\int_{\R^2}
  |\varphi(x,y)| \, dx\, dy.
\]
Hence the sum in the definition \eqref{eq-Tf-defined} of
$\langle T_f, \varphi\rangle$ is a limited hyperreal, whose standard part is a
well defined real number which satisfies
\[
  |\langle T_f, \varphi\rangle | \leq C\int_{\R^2} |\varphi(x,y)| \, dx\, dy.
\]
Therefore $T_f$ is a well defined distribution.

Let $\langle f, \varphi \rangle$ be defined as in
\eqref{eq-function-as-distribution}.  Fix $\varphi$.  We then have
\[
  \int_{\R^2} f(x,y) \varphi(x,y) \,dx \,dy \approx \sum_{(x,y)\in G} f(x,y)
  \varphi(x,y) \,dx \,dy,
\]
which implies that the distribution defined in
\eqref{eq-function-as-distribution} coincides with $T_f$.
\end{proof}

\section{The Cauchy problem for the heat equation}
In this section we recall the definition of distribution solutions to the Cauchy
problem for the heat equation and show how, by solving the finite difference
approximation to the heat equation on a hyperfinite grid, one can construct a
distribution solution to the Cauchy problem.

\subsection{Formulation in terms of distributions}
We consider the Cauchy problem for the linear heat equation $u_{t}=u_{xx}$ with
bounded and continuous initial data $u(x, 0) = u_{0}(x)$.
\begin{definition}
  A distribution $u$ on $\R^{2}$ is a solution to the heat equation
  $u_{t}=u_{xx}$ with initial data $u_{0}$ if $u$ satisfies
  \begin{equation}\label{eq-linear-heat}
    u_t - u_{xx} = u_0(x)\delta(t), \quad x\in \R, t\in\R
  \end{equation}
  in the sense of distributions, and if $u=0$ for $t\leq 0$.
\end{definition}
Equality in the sense of distributions in \eqref{eq-linear-heat} means that both
sides of the equation are to be interpreted as distributions, and that they
should yield the same result when evaluated on any test function
$\varphi\in C_c^\infty(\Omega)$.  To explain this in more detail, recall that
$\delta$ is Dirac's delta function, so that the action of the right hand side in
\eqref{eq-linear-heat} on a test function is
\[
  \langle u_{0}(x) \delta(t), \varphi\rangle \stackrel{\rm def}= \int_\R u_0(x)
  \varphi(x, 0) \,dx.
\]
The definition of distributional derivative \cite[ch.~9]{FollandRealAnalysis}
says that the action of the left hand side in \eqref{eq-linear-heat} is given by
\[
  \langle u_t-u_{xx}, \varphi\rangle = \langle u, -\varphi_t -
  \varphi_{xx}\rangle.
\]
If the distribution $u$ is given by a function $u:\R^2\to\R$ which vanishes for
$t<0$, and is continuous for $t\geq 0$ (so that it has a simple jump
discontinuity at $t=0$)
then we get
\[
  \langle u_t-u_{xx},\varphi\rangle = \int_\R\int_0^\infty u(x, t)
  \bigl\{-\varphi_t - \varphi_{xx}\bigr\}\,dt\,dx.
\]
A piecewise continuous function $u$ therefore satisfies \eqref{eq-linear-heat}
in the sense of distributions if
\begin{equation}\label{eq-heat-weak-solution}
  \int_\R u_0(x)\varphi(x, 0)\,dt
  +
  \int_\R\int_0^\infty u(x, t)\bigl\{\varphi_t+\varphi_{xx}\bigr\}\,dt\,dx
  =0
\end{equation}
for all test functions $\varphi\in C_c^\infty(\R^2)$.  This is one form of the
classical definition of a weak solution to the Cauchy problem.

\subsection{The finite difference equation}
To construct a distribution solution to \eqref{eq-linear-heat} we introduce a
grid with spacing $dx$ and $dt$, and replace the differential equation by the
simplest finite difference scheme that appears in numerical analysis.  If $u$ is
the solution to the differential equation, then we write $U$ for the
approximating solution to the finite difference equation, using the following
common notation for finite differences:
\begin{gather*}
  D_x^+U(x, t) = \frac{U(x+dx, t) - U(x, t)}{dx},\qquad
  D_x^-U(x, t) = \frac{U(x, t) - U(x-dx, t)}{dx}\\[4pt]
  D_t^+U(x, t) = \frac{U(x, t+dt) - U(x, t)}{dt}.
\end{gather*}
See for example, \cite[chapter 1]{LevequeFiniteDifferenceTextbook}.  With this
notation
\[
  D^2_x U(x, t) \stackrel{\rm def}= D_x^+D_x^- U(x, t) = \frac{U(x+dx, t) -
    2U(x,t) + U(x-dx, t)}{(dx)^2}
\]
The operators $D^+_x$, $D^-_x$, and $D_t^+$ all commute.  A finite difference
equation corresponding to the heat equation $u_t=u_{xx}$ is then
$D^+_t U = D^2_xU $, i.e.
\begin{equation}\label{eq-heat-finite-difference}
  \frac{U(x, t+dt) - U(x, t)}{dt} =
  \frac{U(x+dx, t) - 2U(x,t) + U(x-dx, t)}{(dx)^2}
\end{equation}
We can solve this algebraic equation for $U(x, t+dt)$, resulting in
\begin{equation}\label{eq-heat-finite-difference-solved}
  U(x, t+dt) = \alpha U(x-dx,t) + (1-2\alpha)U(x,t) + \alpha U(x+dx,t)
\end{equation}
where
\[
  \alpha \stackrel{\rm def}= \frac{dt}{(dx)^2}.
\]

\subsection{The approximate solution}\label{sec-approx-solution}
Let $N \in \hN$ be an unlimited hyperfinite integer, and assume that $dx$ and
$dt$ are positive infinitesimals.  Assume moreover that $N$ is so large that
both $Ndt$ and $Ndx$ are unlimited hyperreals.  We then consider the hyperfinite
grid
\[
  G_C = \{(m \,dx,n \,dt) \mid m,n \in \hN, |m|+n \leq N\}.
\]
The initial function $u_0:\R\to\R$ extends to an internal function
$u_0:\hR\to\hR$.  By assumption there is a $C\in\R$ such that $|u_0(x)|\leq C$
for all $x\in\R$, so this also holds for all $x\in\hR$.

We define $U:G_C\to\hR$ by requiring
\begin{itemize}
\item $U(x, 0) = u_0(x)$ for all $x$ with $(x,0) \in G_C$, i.e.~for all $x=kdx$
  with $k\in\{-N, \dots, +N\}$.
\item $U$ satisfies \eqref{eq-finite-difference}, or, equivalently,
  \eqref{eq-heat-finite-difference-solved} at all $(x, t) = (mdx, ndt)\in G_C$
  with $|m|+n < N$.
\end{itemize}

\begin{figure}[t]  
  \includegraphics[width=0.7\textwidth]{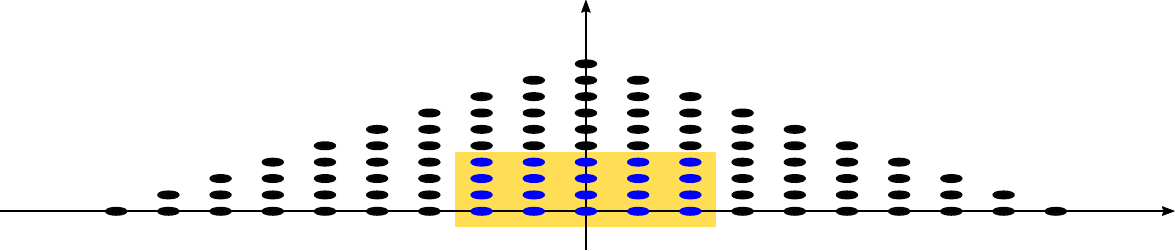}
  \caption{The triangular grid $G_C$; if $U(x, t)$ is known at all grid points
    at the bottom of the triangle, then the finite difference equation
    \eqref{eq-heat-finite-difference} uniquely determines $U(x, t)$ at all other
    grid points.}
  \label{fig-grid}
\end{figure}

\begin{theorem}\label{thm-existence-heat}
Let $U:G_C\to\hR$ be the hyperreal solution of the finite difference scheme
\eqref{eq-heat-finite-difference} with initial values $U(x, 0) = u_0(x)$, and
suppose that $\alpha \leq \frac12$.  Then the expression
\begin{equation}\label{eq-heat-u-def}
  \langle u, \varphi \rangle \stackrel{\rm def}=
  \st \left(\sum U(x,t)\varphi(x,t) \,dx \,dt\right)
  \qquad
  (\varphi\in C_c^\infty(\R^2))
\end{equation}
\[
\]
defines a distribution on $\R^2$ that satisfies \eqref{eq-linear-heat}.
\end{theorem}
To show that this expression does indeed define a distribution we must show that
the $U(x, t)$ are bounded by a standard real number.  This follows from a
discrete version of the maximum principle, which we will again use in
section~\ref{sec-rde}, so we state it in slightly greater generality than needed
in this section.  The lemma explains why we need the condition
$\alpha = dt / (dx)^2 \leq \frac12$, and is well known in numerical analysis as
a necessary condition for stability of the finite difference scheme.

\begin{lemma}[Gronwall type estimate]\label{lem-gronwall}
Let $W:G_C\to\hR$ satisfy $W(x, t)\geq 0$ for all $(x, t)\in G_C$, and suppose
that for some nonnegative $m\in\R$ one has
\[
  W(x, t+dt) \leq \alpha W(x+dx, t) + (1-2\alpha)W(x,t) + \alpha W(x-dx, t) +
  dt\, m W(x, t),
\]
at all $(x,t)\in G_C$.  For each $t = n\,dt$ with $0\leq n\leq N$ consider
\[
  w(t) \stackrel{\rm def}= \max_x W(x, t).
\]
If $0\leq \alpha\leq \frac12$ then
\[
  w(t) \leq e^{mt} w(0).
\]
\end{lemma}
\begin{proof} (Compare \cite[\S7.2--Lemma~I]{john1991partial}.)  The assumption
on $\alpha$ implies that $\alpha\geq 0$ and $1-2\alpha\geq0$. Hence for all $x$
with $(x, t+dt)\in G_C$ we have
\begin{align*}
  W(x, t+dt)
  &= \alpha W(x+dx, t) + (1-2\alpha)W(x,t) + \alpha W(x-dx, t)
    + dt\, m W(x, t) \\
  &\leq \bigl(\alpha+(1-2\alpha)+\alpha + m\, dt\bigr) w(t) \\
  & = (1+m\,dt)w(t).
\end{align*}
Taking the maximum over $x$ we see that
$w(t+dt) \leq (1+m\, dt)w(t) \leq e^{m\, dt}w(t)$.  By induction we then have
for $t=n\,dt$ that $w(t) \leq (e^{mdt})^n w(0) = e^{mt}w(0)$.
\end{proof}

\subsection{Proof of Theorem~\ref{thm-existence-heat}}
\label{sec-proof-distribution-solution}
The relation \eqref{eq-heat-finite-difference-solved} which defines $U(x,t)$
implies that $W(x, t) \stackrel{\rm def}= |U(x, t)|$ satisfies
\begin{align*}
  W(x, t+dt)
  &= \bigl|\alpha U(x-dx,t) + (1-2\alpha)U(x,t) + \alpha U(x+dx,t)\bigr|\\
  &\leq \alpha W(x-dx,t) + (1-2\alpha)W(x,t) + \alpha W(x+dx,t),
\end{align*}
where we have used $\alpha\geq0$ and $1-2\alpha\geq0$.

Since the initial condition is bounded by $|U(x,0)| = |u_0(x)| \leq M$ the
Gronwall type lemma~\ref{lem-gronwall} implies that $|U(x, t)|\leq M$ for all
$(x,t)\in G_C$.  According to Lemma~\ref{lem-Tf-defined} this implies that the
expression \eqref{eq-heat-u-def} does define a distribution $u$ on $\R^2$.

We want to prove that $u$ satisfies the heat equation in the sense of
distributions, i.e.~we want to show for any test function $\varphi$ that
\[
  \langle u_t - u_{xx} ,\varphi \rangle = \langle u_0(x) \delta(t) , \varphi
  \rangle = \int_\R u_0(x) \varphi(x, 0) \, dx.
\]
First, we see from the definition of distributional derivative that
\[
  \langle u_t - u_{xx}, \varphi \rangle = - \langle u, \varphi_t + \varphi_{xx}
  \rangle.
\]
We then have from the definition of $u$ that
\[
  - \langle u , \varphi_t + \varphi_{xx} \rangle \approx \sum_{(x,t)\in G_C}
  -U(x,t) ( \varphi_t (x,t) + \varphi_{xx} (x,t) ) \,dx \,dt \stackrel{\rm def}=
  T.
\]
Using Taylor's formula we replace the partial derivatives of the test function
with its corresponding finite differences, i.e.~we write
$\varphi_t(x,t) = D_t^+ \varphi(x,t) + \varepsilon_t(x,t)$ and
$\varphi_{xx}(x,t) = D_x^2 \varphi(x,t) + \varepsilon_{xx}(x,t)$, where
$\varepsilon_t, \varepsilon_{xx}:G_C\to\hR$ are the infinitesimal error terms in
the Taylor expansion.  Substituting these gives us
\[
  T = \textstyle\sum_{G_C} -U(x,t) \big( D_t^+ \varphi(x,t) + D_x^2 \varphi(x,t)
  + \varepsilon_t + \varepsilon_{xx} \big) \,dx \,dt.
\]
We can split this sum into three parts, $T=T_1+T_2+T_3$, with
\begin{align*}
  T_1 &= \textstyle\sum_{G_C} -U(x,t) D_t^+ \varphi(x,t) \,dx \,dt\\
  T_2 &= \textstyle\sum_{G_C} -U(x,t) D_x^2 \varphi(x,t) \,dx \,dt\\
  T_3 &= \textstyle\sum_{G_C} -U(x,t) (\varepsilon_t + \varepsilon_{xx}) \,dx \,dt
\end{align*}
We will first handle the error term, $T_3$.  Since the test function $\varphi$
has compact support, there exists a real $\ell>0$ such that $\varphi=0$ outside
the rectangle $\Omega = [-\ell, \ell]\times[-\ell, \ell]$.  The errors in the
Taylor expansion therefore also vanish outside of $\Omega$ so that we can write
$T_3$ as
\[
  T_3 = \textstyle\sum_{\Omega\cap G_C} -U(x,t) (\varepsilon_t +
  \varepsilon_{xx}) \,dx \,dt
\]
The key to estimating this sum is that we can estimate all the errors
$\varepsilon_t(x,t)$ and $\varepsilon_{xx} (x,t)$ by one fixed infinitesimal
$\varepsilon>0$ that does not depend on $(x,t)$.  Indeed, the grid $G_C$ is a
hyperfinite internal set, and therefore any internal function such as
$|\varepsilon_t|:G_C\to\hR$ attains its largest value at one of the
$(x,t) \in G_C$, say at $(x_1, t_1)$.  Then
$|\varepsilon_t(x,t)|\leq |\varepsilon_t(x_1, t_1)|$ for all $(x,t)\in G_C$.
Similarly, there is an $(x_2, t_2) \in G_C$ that maximizes
$|\varepsilon_{xx}(x, t)|$.  Now define
\[
  \varepsilon_1 = |\varepsilon_t(x_1, t_1)|, \qquad \varepsilon_2 =
  |\varepsilon_{xx}(x_2, t_2)|.
\]
Then both $\varepsilon_1$ and $\varepsilon_2$ are positive infinitesimals for
which
\[
  |\varepsilon_t(x,t)|\leq \varepsilon_1,\qquad |\varepsilon_{xx}(x, t)| \leq
  \varepsilon_2
\]
holds at all grid points $(x,t) \in G_C$.

If we let $\varepsilon = \max\{\varepsilon_1,\varepsilon_2\}$, then
$|T_3| \leq \sum_{G_C\cap\Omega} 2 U(x,t) \varepsilon \,dx \,dt$.  By the
construction of $U$, we have $|U(x,t)| \leq M$ for all $(x,t)\in G_C$, so we get
\[
  |T_3| \leq \sum_{G_C \cap \Omega} 2M \varepsilon \,dx \,dt \leq 2M \varepsilon
  \,dx \,dt \frac{\ell}{2 \,dx} \frac{\ell}{ \,dt} = M\ell^2 \varepsilon,
\]
which is infinitesimal, so $T_3$ is infinitesimal.

From the definition we have
$T_1 = \sum_{G_C} -U(x,t) (\varphi(x, t+dt) -\varphi(x,t)) \,dx$.  Using the
compact support of $\varphi$, we can then rewrite this sum as
\[
  T_1 = -\sum_{k = -K}^K \sum_{l = 0}^{L+1} U(k \,dx, l \,dt) \bigl\{ \varphi(k
  \,dx, (l+1) \,dt) - \varphi(k \,dx, l \,dt) \bigr\} \,dx,
\]
where $K\,dx \approx \ell$ and $L \,dt \approx \ell$.

Applying summation by parts to this sum we then get
\[
  T_1 = \sum_{k=-K}^K \Bigl\{ U(k \,dx, 0) \varphi( k \,dx,0) + \sum_{l=0}^{L}
  \varphi\bigl(k \,dx, (l+1) \,dt\bigr) D_t^+ U(k \,dx, l \,dt ) \,dt \Bigr\}
  \,dx
\]

Next, for $T_2$, we can split the sum into two parts.
\begin{multline}
  T_2 = \sum_{l = 0}^L \sum_{k = -K-1}^{K+1} -U(k \,dx ,l \,dt) D_x^+ \varphi(k \,dx, l \,dt) \,dt \,dx\\
  + \sum_{l = 0}^L \sum_{k = -K-1}^{K+1} U(k \,dx ,l \,dt) D_x^- \varphi(k \,dx,
  l \,dt) \,dt \, dx
\end{multline}
Applying summation by parts again to both sums we get
\[
  T_2 = -\sum_{l= 0}^L \sum_{k = -K}^K \varphi(k \,dx , l \,dt) D_{x}^2 U(k
  \,dx, l \,dt) \,dx \,dt.
\]
Putting the terms $T_1, T_2, T_3$ all together, we have, because $T_3\approx0$,
\begin{multline}
  T_1 + T_2 + T_3 \approx T_1+T_2 =
  \sum_{k = -K}^K U(k\,dx,0) \varphi(k \,dx, 0) \,dx        \\
  + \sum_{k=-K}^K \sum_{l = 0}^L \varphi(k \,dx, l \,dt) \bigg( D_t^+ U(k\,dx,l
  \,dt) - D_x^2 U(k \,dx, l \,dt) \bigg) \,dx \,dt .
\end{multline}
Since $U$ satisfies the difference equation $D_t^+U = D^2_xU$ at all grid points
this reduces to
\[
  T = T_1+T_2 + T_3 \approx \sum_{k= -K}^K U(k \,dx , 0) \varphi(k \,dx,0) \,dx
  \approx \sum_{k= -K}^K u_0(k \,dx) \varphi(k \,dx,0) \,dx .
\]
Taking the standard part we get the distribution
\[
  \st(T) = \st\left(\sum_{k=-K}^K u_0(k \,dx) \varphi(k \,dx, 0) \,dx \right)
  =\int_{-\ell}^\ell u_0(x)\varphi(x, 0) \,dx = \big\langle u_0(x) \delta(t),
  \varphi \big\rangle.
\]
This completes the proof that
$\langle u_t-u_{xx}, \varphi\rangle = \langle u_0(x)\delta(t), \varphi\rangle$
for all test functions $\varphi$, and thus that $u$ is a distributional solution
of \eqref{eq-linear-heat}.

\subsection{Comments on the proof}
In our construction of solutions to the linear heat equation we completely
avoided estimating derivatives of the approximate solution $U$.  The only
estimate we used was that the approximate solution $U(x,t)$ has the same upper
bound as the given initial function $u_0$.

We assumed that the initial function $u_0$ is continuous.  The one place in the
proof where we needed this assumption was at the end, when we used the fact that
for continuous functions $f:\R\to\R$ one has
\[
  \textstyle\sum_{|k|\leq K} f(k\, dx)\,dx \approx \int_{-\ell}^\ell f(x)\,dx
\]
and applied this to the function $f(x) = u_0(x)\varphi(x, 0)$.

\section{The Cauchy problem for a Reaction Diffusion Equation}
\label{sec-rde}
We consider the reaction diffusion equation
\begin{equation}
  \frac{\partial u}{\partial t} = \frac{\partial^2 u}{\partial x^2} + f(u(x, t)), 
  \qquad (x\in \R, t>0)
  \tag{\ref{eq-rde}}
\end{equation}
with initial condition \eqref{eq-initial-condition}.  It is known that one
cannot expect solutions to exist for all time $t>0$ without imposing some growth
conditions on the nonlinearity $f(u)$.  We will assume that $f$ is a Lipschitz
continuous function, i.e. that for some positive real $K_1$ one has
\begin{equation}\label{eq-f-Lipschitz}
  \forall u, v\in\R : \quad |f(u) - f(v)| \leq K_1|u-v|.
\end{equation}
This implies that $f(u)$ grows at most linearly in $u$:
\begin{equation}\label{eq-f-bound}
  \forall u\in\R : \quad|f(u)| \leq K_0 + K_1|u|
\end{equation}
where $K_0\stackrel{\rm def}=|f(0)|$.

In contrast to the linear heat equation, \eqref{eq-rde} contains the nonlinear
term $f(u)$ which is meaningless if $u$ is an arbitrary distribution.
One can follow the same procedure as in the previous section, i.e.~one can
replace the differential equation by a finite difference scheme on the
hyperfinite grid $G_C$ and construct an approximating solution $U:G_C\to\hR$.
After establishing suitable bounds one can then show that by taking standard
parts as in Lemma~\ref{lem-Tf-defined}, both $U(x, t)$ and $f(U(x,t))$ define
distributions $u$ and $F$ on $\R^2$.  The problem is to give a meaning to the
claim that ``$F = f(u)$,'' because $u$ is merely a distribution and can
therefore not be substituted in a nonlinear function.  In this section we show
how to overcome this problem by adding the assumption that the initial function
is Lipschitz continuous, i.e.
\begin{equation}
  \label{eq-u0-Lipschitz}
  \forall x, y\in\R : \quad |u_0(x) - u_0(y)| \leq L |x-y|
\end{equation}
for some real $L>0$, and showing that the standard part of the approximating
solution $U$ is a continuous function on $\R\times[0, \infty)$.  The
substitution $f(U(x, t))$ is then well defined and we can verify that the
continuous standard function corresponding to $U$ is a distributional solution
of the reaction diffusion equation \eqref{eq-rde}.

\subsection{Weak Solutions to the Reaction Diffusion Equation}
Rather than writing the initial value problem in the distributional form
$u_t-u_{xx}-f(u) = u_0(x)\delta(t)$, we use the integral version
\eqref{eq-heat-weak-solution} of the definition of weak solution.  Thus we
define a \emph{weak solution} to \eqref{eq-rde}, \eqref{eq-initial-condition} to
be a continuous function $u : \R \times [0, \infty)\to\R$ that satisfies
\begin{equation}\label{eq-rde-weak-solution-defined}
  \iint_{\R\times[0, \infty)}
  \bigl \{ u(x,t) (-\varphi_{xx} - \varphi_t) - f(u(x,t)) \varphi \bigr\}
  \,dx \,dt
  = \int_{\R} u_0(x) \varphi(x, 0) \,dx,
\end{equation}
for all test functions $\varphi\in C_c^\infty(\R^2)$.
\begin{theorem}\label{thm-rde-existence-weak-solution}
If $f$ is Lipschitz continuous as in \eqref{eq-f-Lipschitz}, and if the initial
function $u_0$ is bounded by
\[
  \forall x\in\R:\quad |u_0(x)|\leq M,
\]
for some positive real $M$, and if $u_0$ also is Lipschitz continuous, as in
\eqref{eq-u0-Lipschitz}, then the reaction diffusion equation
\eqref{eq-rde},\eqref{eq-initial-condition} has a weak solution.
\end{theorem}
\subsection{Definition of the approximate solution}
To construct the solution we consider the grid $G_C$ as defined in
\S~\ref{sec-approx-solution} with infinitesimal mesh sizes $dx, dt>0$, and
consider the finite difference scheme
\begin{equation}\label{eq-rde-finite-diff}
  D^+_t U(x,t) = D^2_xU (x, t) + f(U(x,t)).
\end{equation}
Solving for $U(x, t+dt)$ we get
\begin{multline}
  \label{eq-rde-finite-diff-solved}
  U(x,t+\,dt)
  = \alpha U(x+ \,dx, t) + (1-2 \alpha) U(x,t) + \alpha U(x - \,dx, t)\\
  + dt\, f(U(x,t)),
\end{multline}
where, as before, $\alpha = dt/(dx)^2$.  We extend the continuous function $u_0$
to an internal function $u_0:\hR\to\hR$, and specify the initial conditions
$U(x, 0) = u_0(x)$ for $x=m\,dx$, $m=-N, \dots, +N$.  The finite difference
equation \eqref{eq-rde-finite-diff-solved} then determines $U(x, t)$ for all
$(x,t)\in G_C$.

We now establish a number of \emph{a priori} estimates for the approximate
solution $U$ that will let us verify that its standard part is well defined and
that it is a weak solution of the initial value problem.

\subsection{Boundedness of the approximate solution}
First we establish a bound for $|U(x, t)|$.
\begin{lemma} For all $(x, t)\in G_C$ we have
\begin{equation} \label{eq-rde-U-upperbound} |U(x, t)| \leq e^{K_1t}M +
  \frac{K_0}{K_1}(e^{K_1t}-1).
\end{equation}
\end{lemma}
\begin{proof}
Using \eqref{eq-f-bound}, i.e.~$|f(u)|\leq K_0+K_1|u|$, we get
\begin{multline}
  |U(x, t+dt)|\leq
  \alpha |U(x+dx, t)| + (1-2\alpha)|U(x, t)| +\alpha |U(x-dx, t)|\\
  +dt (K_0 + K_1|U(x, t)|).
\end{multline}
In terms of $M(t) = \max_x|U(x, t)|$ this implies
\[
  M(t+dt) \leq M(t) +dt (K_0+K_1M(t)) = (1+K_1dt)M(t) + K_0dt.
\]
Setting $t=n\,dt$ we see that this is an inequality of the form
$M_n \leq a M_{n-1} +b$ with $M_n = M(ndt)$.  By induction this implies that
\begin{align*}
  M(t) &= M(ndt) \leq (1+K_1dt)^n M(0) + \frac{(1+K_1dt)^n-1}{1+K_1dt -1}K_0dt\\
       &\leq e^{K_1t}M(0) + \frac{K_0}{K_1}(e^{K_1t}-1).
\end{align*}
Since $M(0)= M$ this proves the \eqref{eq-rde-U-upperbound}.
\end{proof}

\subsection{Lipschitz continuity in space of the approximate solution}
We now show that $U(x, t)$ is Lipschitz continuous in the space variable.
\begin{lemma}
For any two points $(x, t), (x', t)\in G_C$ we have
\begin{equation}\label{eq-rde-U-space-lipschitz}
  |U(x, t) - U(x', t)| \leq L e^{K_1 t} |x-x'|,
\end{equation}
where $L$ is the Lipschitz constant for the initial function $u_0$, as
in~\eqref{eq-u0-Lipschitz}.
\end{lemma}
\begin{proof}
Let
\[
  V(x,t) \stackrel{\rm def}= \frac{U(x+\,dx,t) - U(x,t)}{dx} = D_x^+U(x, t).
\]
Applying $D_x^+$ to both sides of the equation \eqref{eq-finite-difference} for
$U$, and using the definition of $V$ and commutativity of the difference
quotient operators, we find
\[
  D_t^+ V = D_x^2 V + D_x^+ f(U).
\]
Solving for $V(x, t+dt)$ we find
\[
  V(x, t+dt) = \alpha V(x+dx, t) + (1-2\alpha)V(x, t) +\alpha V(x-dx, t) + D_x^+
  f(U)
\]
Examining $D_x^+ f(U)$, we have
\begin{multline*}
  |D_x^+ f(U)| = \frac{|f(U(x+\,dx, t)) - f(U(x,t))|}{dx} \leq\\
  \frac{K_1\,|U(x + \,dx , t) - U(x,t)|}{dx} = K_1|V(x,t)|,
\end{multline*}
so that
\[
  |V(x, t+dt)|\leq \alpha |V(x+dx, t)| + (1-2\alpha)|V(x, t)| +\alpha|V(x-dx,
  t)| + K_1 dt |V(x, t)|.
\]
Using Gronwall's Inequality on $\max\limits_x V$, we get the inequality
\[
  \max_x |V(x,t)| \leq e^{K_1t} \max_x |V(x,0)|.
\]
The initial condition $u_0$ satisfies $|u_0(x)-u_0(x')| \leq L|x-x'|$ for all
$x, x'\in\R$, and therefore the extension of $u_0$ to the hyperreals satisfies
this same inequality.  Therefore $|V(x, 0)|\leq L$ for all grid points $(x, 0)$,
and thus we have $|V (x, t)|\leq Le^{K_1 t}$.  This implies
\eqref{eq-rde-U-space-lipschitz}.
\end{proof}

\subsection{H\"older continuity in time of the approximate solution}
\begin{lemma}
Given any real $\bar t>0$ we have for any two grid points
$(x_0, t_0), (x_0, t_1)\in G_C$ with $0\leq t_0 \leq t_1 \leq \bar t$
\begin{equation} \label{eq-rde-U-time-holder} |U(x_0, t_1)-U(x_0, t_0)| \leq
  C\sqrt{t_1-t_0}
\end{equation}
where $C$ is a constant that only depends on $\bar t$, $K_0$, $K_1$, $L$, and
$M$.
\end{lemma}
\begin{proof}
We begin by oberving that $f(U(x,t))$ is bounded on the time interval we are
considering.  Indeed, for $t\leq \bar t$ we have shown for all $(x,t)\in G_C$
that
\[
  |U(x, t)| \leq A_0 \stackrel{\rm def}= e^{K_1 \bar t}M + \frac{K_1}{K_0}
  \left(e^{K_1\bar t}-1\right),
\]
while the Lipschitz condition for $f$ implies
$|f(U(x,t))| \leq K_0+K_1|U| \leq K_0+K_1A_0$.  So if we set $A=K_0+K_1A_0$,
then we have
\begin{equation} \label{eq-fU-bound} 
  \forall (x, t)\in G_C \text{ with }t\leq
  \bar t:\quad |f(U(x, t))| \leq A.
\end{equation}

Next, we construct a family of upper-barriers for $U$ using parabolas.  In
particular, for any real $a, b, c>0$ we consider
\[
  \bar U(x, t) \stackrel{\rm def}= U(x_0,t_0) + a(t-t_0) + \frac b2(x-x_0)^2+c.
\]
For any $b>0$ we will find $a, c>0$ so that $\bar U$ is an upper barrier, in the
sense that
\begin{equation} \label{eq-upper-barrier} D_t^+\bar U - D^2_x\bar U \geq A + 1,
\end{equation}
and
\begin{equation} \label{eq-upper-barrier-initial-data} \bar U(x, 0) > U(x, 0)
  \text{ for all } (x,0)\in G_C.
\end{equation}
A direct computation shows that $D_t^+ \bar U - D^2_x \bar U = a-b$, so for
given $b$ we choose $a= b+A+1$ and \eqref{eq-upper-barrier} will hold.

To satisfy \eqref{eq-upper-barrier-initial-data} we use
\eqref{eq-rde-U-space-lipschitz}, i.e.~that $U(x, t)$ is Lipschitz continuous
with Lipschitz constant $\bar L \stackrel{\rm def}= e^{K_1\bar t}L$:
\[
  U(x, t)\leq U(x_0, t_0) + \bar L |x-x_0| \leq U(x_0, t_0) + \frac 2b(x-x_0)^2
  + \frac{\bar L^2}{2b}.
\]
If we choose $c>\bar L^2/2b$, e.g.~$c=\bar L^2/b$, then our upper barrier
$\bar U$ also satisfies \eqref{eq-upper-barrier-initial-data}.

Next, we apply a maximum principle argument to compare $U$ and $\bar U$.
Consider $W(x, t)= U(x, t) -\bar U(x, t)$.  Then we have shown that $W(x, 0)<0$
for all $x$, and $D_t^+W- D_x^2W <0$, which implies
\[
  W(x, t+dt) < \alpha W(x-dx, t) + (1-2\alpha) W(x, t) + \alpha W(x+dx, t)
\]
for all $(x, t)$ for which $(x\pm dx, t) \in G_C$.  By induction we get
$W(x, t)<0$ for all $(x,t) \in G_C$.  In particular $U(x_0, t) < \bar U(x_0, t)$
for all $t>t_0$, i.e.~we have shown
\[
  U(x_0, t_1) < U(x_0, t_0) + (b+A+1)(t_1-t_0) + \frac{\bar L^2}{b}.
\]
This upper bound holds for any choice of $b>0$.  To get the best upper bound we
minimize the right hand side over all $b>0$.  The best bound appears when
$b = \bar L/\sqrt{t_1-t_0}$.  After some algebra one then finds
\[
  U(x_0, t_1)-U(x_0, t_0) < (A+1) (t_1-t_0) + 2 \bar L\sqrt{t_1-t_0}.
\]
Finally, using
$t_1-t_0 = \sqrt{t_1-t_0}\sqrt{t_1-t_0} \leq \sqrt{\bar t}\sqrt{t_1-t_0}$ we get
\[
  U(x_0, t_1)-U(x_0, t_0) < \bigl((A+1)\sqrt{\bar t} + 2 \bar L\bigr)
  \sqrt{t_1-t_0}.
\]
This proves the upper bound in \eqref{eq-rde-U-time-holder}.  To get the
analogous lower bound one changes the signs of the coefficients $a,b,c$ which
will turn $\bar U$ into a lower barrier.  After working through the details one
finds the appropriate lower bound.
\end{proof}
Now that we have Lipschitz in space and $\textrm{H\"older}$ in time, we also
have for any pair of points $(x, t), (y, s)\in G_C$ that
\begin{multline} \label{eq-U-continuous}
  |U(x,t) - U(y,s)| \leq |U(x,t) - U(y,t)| + |U(y,t) - U(y,s)| \\
  \leq L|x-y| + C \sqrt{|t-s|} .
\end{multline}

\subsection{Definition of the weak solution}
So far we have been establishing estimates for the solution $U$ to the finite
difference scheme.  It is worth pointing out that a standard existence proof
would have required exactly the same estimates.  At this point however, the
standard and nonstandard proofs diverge.

For $(x, t)\in\R\times[0, \infty)$ we choose $(\tilde x, \tilde t)\in G_C$ with
$x\approx \tilde x$ and $t\approx \tilde t$, and then define
$u(x, t) = \st (U(\tilde x, \tilde t))$.  The continuity property
\eqref{eq-U-continuous} of the approximate solution $U$ implies that the value
of $\st (U(\tilde x, \tilde t))$ does not depend on how we chose the grid point
$(\tilde x, \tilde t)$, for if $(\hat x, \hat t)\in G_C$ also satisfied
$\hat x\approx x$, $\hat t\approx t$, then $\tilde x\approx \hat x$ and
$\tilde t\approx \hat t$, so that
$U(\tilde x, \tilde t)\approx U(\hat x, \hat t)$.  It follows directly that the
function $u:\R\times[0, \infty)\to\R$ is well defined, that it satisfies the
continuity condition \eqref{eq-U-continuous}, and that it satisfies the same
bounds as in \eqref{eq-rde-U-upperbound}.

By the transfer principle the standard function $u$ extends in a unique way to
an internal function $\hR\times{}^*[0, \infty) \to \hR$.  It is common practice
to abuse notation and use the same symbol $u$ for the extension.  The extended
function satisfies the same continuity condition \eqref{eq-U-continuous}.
\begin{lemma}
If $(x, t)\in G_C$ is limited, then $u(x, t)\approx U(x, t)$.
\end{lemma}
\begin{proof}
If $(x, t)$ is limited, then $x'=\st (x)$ and $t'=\st (t)$ are well defined real
numbers.  By continuity of both $u$ we have $u(x, t)\approx u(x', t')$.  By
definition of $u$ it follows from $x'\approx x$ and $t'\approx t$ that
$u(x', t') \approx U(x, t)$.  Combined we get $u(x, t)\approx U(x, t)$.
\end{proof}

\subsection{Proof that $u$ is a weak solution}
We will now show that $u$ is a weak solution whose existence is claimed in
Theorem~\ref{thm-rde-existence-weak-solution}, i.e.~we verify that $u$ satisfies
\eqref{eq-rde-weak-solution-defined} for any test function
$\varphi\in C_c^\infty(\R^2)$.

Since $\varphi$ has compact support, there is a positive real $\ell$ such that
$\varphi(x, t)=0$ outside the square $[-\ell, \ell]\times [-\ell, \ell]$.  We
therefore have to verify
\[
  \int_0^\ell\int_{-\ell}^\ell \left \{ u(x,t) (-\varphi_{xx} - \varphi_t) -
    f(u(x,t)) \varphi \right\} \,dx \,dt = \int_{-\ell}^\ell u_0(x) \varphi(x,
  0) \,dx.
\]
Since the integrands are continuous functions we only make an infinitesimal
error when we replace the two Riemann integrals by Riemann sums over the part of
the grid $G_C$ that lies within the square $[-\ell, \ell]\times[-\ell, \ell]$.
Thus we must prove
\begin{equation} \label{eq-rde-to-prove} \sum_{(x,t)\in G_C}\bigl \{ u(x,t)
  (-\varphi_{xx} - \varphi_t) - f(u(x,t)) \varphi \bigr\} \,dx \,dt \approx
  \sum_{(x, 0)\in G_C} u_0(x) \varphi(x, 0) \,dx.
\end{equation}

We now intend to replace $u$ by $U$, and the derivatives of $\varphi$ by the
corresponding finite differences.  In doing so we make errors that we must
estimate.  Let $G_{C\ell} = G_C\cap [-\ell, \ell]^2 $, so that the only nonzero
terms in the two sums come from terms evaluated at points in $G_{C\ell}$.  The
intersection of internal sets is again internal, so the set $G_{C\ell}$ is
internal and hyperfinite.

For each $(x, t)\in G_{C\ell}$ the quantities
\[
  |u(x, t) - U(x, t)|,\quad |\varphi_t(x, t)- D_t^+\varphi(x, t)|, \text{ and }
  |\varphi_{xx}(x, t)-D^2_x\varphi(x, t)|
\]
are infinitesimal.  Since they are defined by internal functions, one of the
numbers in the hyperfinite set
\[
  \left\{ |u(x, t) - U(x, t)|, |\varphi_t(x, t)- D_t^+\varphi(x, t)|,
    |\varphi_{xx}(x, t)-D^2_x\varphi(x, t)| : (x, t)\in G_{C\ell} \right\}
\]
is the largest.  This number, which we call $\varepsilon$, is again
infinitesimal.  Therefore we have
\begin{equation}\label{eq-rde-error-estimate}
  \max_{G_{C\ell}}\left\{
    |u(x, t) - U(x, t)|, |\varphi_t(x, t)- D_t^+\varphi(x, t)|,
    |\varphi_{xx}(x, t)-D^2_x\varphi(x, t)| \right\}
  \leq \varepsilon
\end{equation}
for some infinitesimal $\varepsilon>0$.

The remainder of the argument is very similar to our proof in
\S~\ref{sec-proof-distribution-solution} that the distribution $u$ defined was a
distribution solution to the linear heat equation.  Namely, if we replace $u$ by
$U$ and derivatives of $\varphi$ by finite differences of $\varphi$ in
\eqref{eq-rde-to-prove}, then \eqref{eq-rde-error-estimate} implies that we only
make an infinitesimal error on both sides.  We therefore only have to prove
\[
  \sum_{(x,t)\in G_C}\bigl \{ U(x,t) (-D^2_x\varphi - D^+_t\varphi_t) -
  f(u(x,t)) \varphi \bigr\} \,dx \,dt \approx \sum_{(x, 0)\in G_C} u_0(x)
  \varphi(x, 0) \,dx.
\]
This follows after applying summation by parts, and using the finite difference
equation \eqref{eq-rde-finite-diff} satisfied by $U$.  This completes the
existence proof.


%


\bibliographystyle{plain} \bibliography{refs.bib}

\end{document}